\documentclass[11pt]{amsart}
\usepackage{amssymb}
\usepackage{amscd}
\usepackage{verbatim}
\usepackage{amssymb,amsfonts,pstricks,epsf}
\usepackage{graphicx}
\usepackage{epsfig}
\usepackage{color}
\usepackage{stmaryrd}
\usepackage{mathrsfs}
\usepackage{marginnote}

\newtheorem{theorem}{Theorem}[section]

\newtheorem{proposition}[theorem]{Proposition}
\newtheorem{corollary}[theorem]{Corollary}

\definecolor{plum}{rgb}{1.0, 0.0, 1.0}

\DeclareMathOperator{\csch}{csch}

\begin{document}

\title{Torsional rigidity and isospectral planar sets}

\author{Joseph Comer and Patrick McDonald}

\address{Hughes Research Laboratory}

\email{joseph.comer@ncf.edu}

\address{Division of Natural Science, New College of Florida, Sarasota, FL 34243}

\email{mcdonald@ncf.edu}

\date{\today}

\begin{abstract}
We prove that a certain pair of isospectral planar sets are distinguished by torsional rigidity.
\end{abstract}

\keywords{torsional rigidity, heat content, Dirichlet Problem}

\subjclass[2010]{58J50, 58J35, 35P10}

\maketitle

\section{Introduction}

Let $\Omega\subset {\mathbb R}^2$ be a planar domain with compact closure and piecewise smooth boundary.  Let $\Delta$ be the Laplace operator and suppose $u:\Omega \to {\mathbb R}$ satisfies 
\begin{eqnarray*}
\Delta u &= &-1 \hbox{ on } \Omega \\
u &=& 0 \hbox{ on } \textup{boundary}(\Omega).
\end{eqnarray*}
The {\it torsional rigidity} of $\Omega$ is the real number $T(\Omega)$ defined by 
\begin{equation}\label{Eqn:TorsionalRigidityDef}
T(\Omega)  =  \int_{\Omega} u(x,y) dxdy
\end{equation}
where $dxdy$ is the usual area form.  Originally defined in the context of the theory of elastic bodies where it is used to quantify the torque required to produce a given amount of twist in a homogeneous cylinder of unit height and cross-section $\Omega,$ torsional rigidity played a fundmanetal role in engineering mechanics throughout the nineteenth and twentieth centuries. Important early work of de Saint-Venant, in particular his 1856 formulation of the torsion problem, served to broaden interest in torsional rigidity as a topic for geometric analysis, broadly construed.  The first rigorous solution of the torsion problem appears in a 1948 paper of P\'olya \cite{Po1} where he uses symmetrization techniques to simultaneously address the torsion problem and the Rayleigh conjecture, thus drawing a connection between geometric and analytic properties of torsional rigidity and the principal Dirichlet eigenvalue.  P\'olya and Szeg\"o \cite{PS1} further developed these connections, establishing inequalities involving area, torsional rigidity, and the principal Dirichlet eigenvalue (denoted by $\lambda_1(\Omega)$) that include
\[
T(\Omega)\lambda_1(\Omega) \leq \textup{Area}(\Omega).
\]
The work of P\'olya and Szeg\"o spurred interest in elucidating the relationship between torsional rigidity, Dirichlet spectrum and the geometry of a given domain, and the associated literature now fills many journal pages.  Our result adds to this literature.  We prove:

\begin{theorem}\label{Thm:DifferentTorsion}There are piecewise smoothly bounded open planar sets $C_1$ and $C_2$ that are Dirichlet isospectral, but for which $T(C_1) \neq T(C_2).$
\end{theorem}

The sets we study were first constructed by Chapman \cite{Ch1}, who established their isospectrality.  In \cite{BDK}, van den Berg, Dryden and Kappeler prove that these sets have different heat content.  Because torsional rigidity is the first moment of heat content (see, for example, \cite{MM1}), we recover their result for the Chapman pairs in question.  

The Chapman domains we study are disjoint unions of planar polygons and it is natural to question whether there are examples of isospectral pairs of connected open sets that are distinguished by torsional rigidity.  It is an observation of Gilkey \cite{Gi1} that isospectral pairs that arise via a Sunada construction must have identical heat content.  Our observation that torsional rigidity can be identified with the first moment of heat content then implies that torsional rigidity can not distinguish such isospectral pairs.  We know of no other connected, isospectral, non-isometric pairs of planar sets.  There are examples of isospectral pairs of connected combinatorial graphs distinguished by torsional rigidity \cite{MM2} and examples of isospectral pairs of connected quantum graphs distinguished by torsional rigidity \cite{CKM}. 

The proof of Theorem \ref{Thm:DifferentTorsion} involves an explicit calculation of the torsional rigidity for the domains in question.  In the section that follows we fix notation, develop the machinery required for our calculations, and compute torsional rigidity for rectangles and right isosceles triangles.  While formulae for torsional rigidity involving rectangles and certain triangles can be found in the literature (see for example \cite{TG1} and references therein), choices involving normalizing constants and differences in terminology suggest that there is value in making the presentation as self-contained as possible.  In the final section of the paper we formally define the Chapman pairs referenced in our theorem (see figure \ref{Fig:Chapman}) and provide the computation required for the proof.

\section{Torsional rigidity for rectangles and right isosceles triangles}      

Let $\Omega \subset {\mathbb  R}^2$ be a rectangle of length $L$ and height $H.$  Let $\Delta = \frac{\partial^2}{\partial x^2} + \frac{\partial^2}{\partial y^2}$ be the Laplace operator and recall that the Dirichlet spectrum is the discrete collection of real numbers $\lambda$ for which there is a nontrivial solution of 
\begin{eqnarray*}
\Delta u +\lambda u & = & 0 \hbox{ on } \Omega   \\
u & = & 0 \hbox{ on } \partial \Omega. 
\end{eqnarray*}
For the rectangle $\Omega,$ the Dirichlet spectrum of $\Omega$ is parametrized by pairs of positive integers:
\[
\lambda_{j,k} = \pi^2\left(\left(\frac{j}{L}\right)^2 +\left(\frac{k}{H}\right)^2 \right).
\]
The $L^2$-normalized eigenfunction associated to $\lambda_{j,k}$ is 
\[
\phi_{j,k}(x,y) = \frac{2}{\sqrt{LH}} \sin\left(\frac{j\pi x}{L}\right) \sin\left(\frac{k\pi y}{H}\right).
\]
The Green's function associated to $\Omega$ can be written formally in terms of the spectral data:
\begin{equation}\label{Eqn:GF}
G(x,y,x^\prime,y^\prime) = \sum_{j, k \geq 1}^\infty \frac{1}{\lambda_{j,k}}\phi_{j,k}(x,y)\phi_{j,k}(x^\prime,y^\prime).
\end{equation}
The solution of the Poisson problem 
\begin{eqnarray*}
\Delta u & = & -1 \hbox{ on } \Omega   \\
u & = & 0 \hbox{ on } \partial \Omega 
\end{eqnarray*}
can be written in terms of the Green's function:
\[
u(x,y) = \int_\Omega G(x,y,x^\prime,y^\prime) dx^\prime dy^\prime.
\]
From (\ref{Eqn:GF}) it follows that the torsional rigidity of $\Omega$ can be expressed as 
\begin{equation}\label{Eqn:TR2}
T(\Omega) = \sum_{j,k \geq 1} a_{j,k}^2 \frac{1}{\lambda_{j,k}}
\end{equation}
where the constants $a_{j,k}$ are obtained by integrating the normalized eigenfunctions:
\[
a_{j,k} = \int_\Omega \phi_{j,k}(x,y) dx dy.
\]

We can compute the coefficients occuring (\ref{Eqn:TR2}):
\[
a_{j,k} = \left\{ \begin{array}{ll}
                 4 \frac{2}{\sqrt{LH}} \frac{LH}{\pi^2 jk} & \hbox{ if $j$ and  $k$ are both odd} \\
                 0 & \hbox{ else} \end{array} \right. 
\]
and thus (see page 108 of \cite{PS1}):
\[
T(\Omega) = \frac{4^3 HL}{\pi^{6}}\sum_{j,k \textup{ odd}} \frac{1}{j^2k^2} \left(\frac{1}{\left(\frac{j}{L}\right)^2 +\left(\frac{k}{H}\right)^2}\right).
\]
Setting $x= \frac{H}{Lk}$ gives
\[
T(\Omega) = \frac{4^3 H^{3}L}{\pi^{6}}\sum_{k \textup{ odd}} \frac{1}{k^{4}}\sum_{j \textup{ odd}} \frac{1}{j^2}\frac{1}{1+j^2 x^2}.
\]
There is a partial fraction decomposition:
\[
T(\Omega) =  \frac{4^3 H^{3}L}{\pi^{6}}\sum_{k \textup{ odd}} \frac{1}{k^{4}}\left( \sum_{j \textup{ odd}} \frac{1}{j^2} -x^2\frac{1}{1+j^2x^2}  \right).
\]

Note that 
\begin{eqnarray*}
\sum_{j \textup{ odd}} \frac{1}{j^2} & = & \sum_{j \geq 1} \frac{1}{j^2} - \sum_{j \textup{ even}} \frac{1}{j^2} \\
& = & \left(1 - \frac{1}{2^2}\right)\zeta(2)
\end{eqnarray*}
where $\zeta(s)$ denotes the Riemann zeta function evaluated at $s.$  Setting 
\[
Z_r = \left(1-\frac{1}{2^{4}}\right)\left(1 - \frac{1}{2^2}\right)\zeta(4)\zeta(2) = \left(\frac{1}{12}\right) \frac{\pi^6}{ 2^6},
\]
we have 

\begin{equation}\label{Eqn:TnRep3}
T(\Omega) =  \frac{4^3 H^{3}L}{\pi^{6}}\left(Z_r - \left(\frac{H}{L}\right)^2 \sum_{k \textup{ odd}} \frac{1}{k^{6}}\left( \sum_{j \textup{ odd}} \frac{1}{1+j^2x^2}  \right) \right).
\end{equation}

Using contour integration it is an easy exercise to find a closed form for the series involving the index $j:$  

\[
\sum_{j \geq 1} \frac{1}{1+j^2x^2} = -\frac12 +  \frac{\pi \coth\left(\frac{\pi}{x}\right)}{2x}.
\]

Setting $\beta = \pi \frac{H}{L}$ and $\gamma=\pi \frac{L}{H},$ we have 
\begin{eqnarray*}
\sum_{j \textup{ odd}} \frac{1}{1+j^2x^2} & = & \frac{\pi \coth\left(\frac{\pi}{x}\right)}{2x} - \frac{\pi \coth\left(\frac{\pi}{2x}\right)}{4x} \\
 & = & \frac{\pi}{2x}\left(\coth(k\gamma) - \frac{\coth\left(\frac{k\gamma}{2}\right)}{2}\right).
\end{eqnarray*}
Using the identity $\coth(\frac{\theta}{2}) = \coth(\theta) + \csch(\theta)$ we have
\[
\sum_{j \textup{ odd}} \frac{1}{1+j^2x^2} = \frac{\pi}{2x}\frac12 \left(\coth(k\gamma) - \csch(k\gamma) \right) .
\]
Finally, the identity $\tanh(\frac{\theta}{2}) = \coth(\theta) - \csch(\theta)$ gives
\begin{equation}\label{Eqn:jsum2}
\sum_{j \textup{ odd}} \frac{1}{1+j^2x^2} = \frac{\pi}{4}\left(\frac{Lk}{H}\right) \tanh\left(\frac{k\gamma}{2}\right) .
\end{equation}
Combining (\ref{Eqn:TnRep3}) and (\ref{Eqn:jsum2}) establishes

\begin{proposition}\label{Thm:FirstMomentR}Suppose $\Omega$ is a rectangle of length $L$ and height $H.$ Let $\beta = \pi \frac{H}{L}$ and let $\gamma=\pi \frac{L}{H}.$  Then 
\begin{equation}\label{Eqn:FirstMoment}
T(\Omega) = \frac{4^3H^3L}{\pi^{6}}\left(Z_r - \frac{1}{4} \beta \sum_{k \textup{ odd}} \frac{1}{k^5} \tanh\left(\frac{k\gamma}{2}\right) \right)
\end{equation}
where
\[
Z_r = \left(\frac{1}{12}\right) \frac{\pi^6}{ 2^6}.
\]
\end{proposition}

As an aside, series of the form appearing in the proof of Proposition \ref{Thm:FirstMomentR} occur in a variety of mathematical context and are much-studied.  For our purposes, invariance of torsional rigidity under the exchange of length and height leads to a rational expression in $L$ and $H$ for the expression 
\[
\beta^2 \sum_{k \textup{ odd}} \frac{1}{k^5} \tanh\left(\frac{k\gamma}{2}\right) - \gamma^2 \sum_{k \textup{ odd}} \frac{1}{k^5} \tanh\left(\frac{k\beta}{2}\right). 
\]
This observation, first made by de Saint-Venant \cite{SV1}, is a special case of a general phenomena best understood from the viewpoint of modular forms (see \cite{Be1} and references therein).

The techniques used to establish Proposition \ref{Thm:FirstMomentR} can be used to derive expressions for the torsional rigidity for right isosceles triangles\footnote{This is true for any polygon for which a complete collection of eigenfunctions consisting of trigonometric polynomials exists.  Such polygons have been classified by McCartin \cite{Mc1}.}.  To see that this is the case, suppose $\Omega$ is a right isosceles triangle of base length $L.$  Then, using reflection and transplantation, the Dirichlet eigenvalues of $\Omega $ are given by
\[
\lambda_{j,k} = \frac{\pi^2}{L^2}(j^2 + k^2) \hbox{ where } j < k.
\]
Assuming $\Omega$ is parametrized by $0\leq x\leq L,$ $0\leq y \leq L-x,$ the associated normalized eigenfunctions are\footnote{The spectral resolution for the Dirichlet Laplacian for right isosceles triangles can be found in \cite{P1}; see also the discussion in \cite{Mc1}.}   
\small
\[
\phi_{j,k}(x,y)  = \frac{2}{L}\left(\sin\left(\frac{j\pi x}{L}\right)\sin\left(\frac{k\pi y}{L}\right) +(-1)^{j+k+1}\sin\left(\frac{k\pi x}{L}\right)\sin\left(\frac{j\pi y}{L}\right)\right).
\]
\normalsize
As in the case of rectangles, the torsional rigidity of $\Omega$ is given by integrating the Green's function:
\[
T(\Omega) = \sum_{k\geq 1} \sum_{j< k} a_{j,k}^2 \frac{1}{\lambda_{j,k}}
\]
where 
\[
a_{j,k}^2 = \left\{\begin{array}{ll}
                   \frac{4^2 L^2}{\pi^4}\frac{1}{j^2k^2} & \textup{ if } j \textup{ even and } k \textup{ odd or } j \textup{ odd and } k \textup{ even } \\
                   0 & \textup{ else. }
                   \end{array} \right.
\]
Writing $F(j,k)= \frac{1}{j^2k^2}\frac{1}{j^2+k^2}$ we have
\[
T(\Omega) =  \frac{4^2L^4}{\pi^6} \left( \sum_{k\textup{ even}} \sum_{\stackrel{j< k}{j \textup{ odd}} } F(j,k) + \sum_{k\textup{ odd}} \sum_{\stackrel{j< k}{j \textup{ even}} } F(j,k) \right). 
\]
Because the coefficient $a_{j,k}^2$ are only nonzero for positive integer pairs with different parity and the expression being summed is symmetric in $k$ and $j,$ we have
\begin{eqnarray*}
T(\Omega) & = & \frac12 \frac{4^2L^4}{\pi^6} \left( \sum_{k\textup{ even}} \sum_{j \textup{ odd}}  F(j,k)  +  \sum_{k\textup{ odd}} \sum_{j \textup{ even}}  F(j,k) \right).  \label{Eqn:IR4} 
\end{eqnarray*}

The sum $ {\mathcal S} = \sum_{k\textup{ even}} \sum_{j \textup{ odd}}  F(j,k)  +  \sum_{k\textup{ odd}} \sum_{j \textup{ even}}  F(j,k)$ is a sum over a subset of the lattice of positive integer pairs: those pairs with integers of different parity.  We sum over the entire lattice and throw out the integer pairs with the same parity:
\begin{eqnarray*}
{\mathcal S} & = &\sum_{j,k \geq 1}  F(j,k)  - \sum_{j, k\textup{ even}}   F(j,k) -  \sum_{j, k\textup{ odd}}  F(j,k) \\
 & = & S_a-S_e-S_o.
 \end{eqnarray*}
 
We compute each sum.
 
\begin{eqnarray}
S_a & = & \sum_{k\geq 1} \frac{1}{k^4}\sum_{j\geq 1} \frac{1}{j^2} \frac{1}{1+j^2(1/k)^2} \nonumber \\ 
 &=& \sum_{k\geq 1} \frac{1}{k^4}\sum_{j\geq 1} \frac{1}{j^2} + \frac{-(1/k)^2}{1+j^2(1/k)^2} \nonumber \\
 &=& \zeta(4)\zeta(2) - \sum_{k\geq 1} \frac{1}{k^6}\sum_{j\geq 1} \frac{1}{1+j^2(1/k)^2} \nonumber \\
  &=& \zeta(4)\zeta(2) - \sum_{k\geq 1} \frac{1}{k^6} \left(-\frac12 +\frac{\pi \coth(k\pi)}{2(1/k)} \right) \nonumber \\
   &=& \zeta(4)\zeta(2) +\frac12\zeta(6) - \frac{\pi}{2}\sum_{k\geq 1} \frac{1}{k^5} \coth(k\pi). \label{Eqn:S}
\end{eqnarray}

To compute $S_e,$ note that $F(2j,2k) = \frac{1}{2^6}F(j,k):$
\begin{equation}\label{Eqn:Se}
S_e = \frac{1}{2^6} S_a.
\end{equation}

Finally, we compute $S_o:$
\begin{eqnarray}
S_o & = & \sum_{k \textup{ odd}} \frac{1}{k^4}\sum_{j\textup{ odd}} \frac{1}{j^2} + \frac{-(1/k)^2}{1+j^2(1/k)^2} \nonumber \\ 
 &=& \left(1-\frac{1}{2^4}\right)\left(1-\frac{1}{2^2}\right)\zeta(4)\zeta(2) - \sum_{k\textup{ odd}} \frac{1}{k^6}\sum_{j\textup{ odd}} \frac{1}{1+j^2(1/k)^2} \nonumber \\
  &=& \left(1-\frac{1}{2^4}\right)\left(1-\frac{1}{2^2}\right)\zeta(4)\zeta(2) - \sum_{k\textup{ odd}} \frac{1}{k^6} \left(-\frac12 + \frac{\pi \coth(k\pi)}{2(1/k)}\right)  \nonumber \\
  &  &  + \sum_{k\textup{ odd}} \frac{1}{k^6} \left( -\frac12 + \frac{\pi \coth\left(\frac{k\pi}{2}\right)}{4(1/k)} \right)   \nonumber \\
 &=&C - \frac{\pi}{2} \sum_{k\textup{ odd}} \frac{1}{k^5} \coth(k\pi) +  \frac{\pi}{4} \sum_{k\textup{ odd}} \frac{1}{k^5} \coth\left(\frac{k\pi}{2}\right) \label{Eqn:So}
\end{eqnarray}
where $C= \left(1-\frac{1}{2^4}\right)\left(1-\frac{1}{2^2}\right)\zeta(4)\zeta(2).$  Let $Z_t$ be defined by
\begin{eqnarray}
Z_t& =& \left(\left(1-\frac{1}{2^6}\right) -  \left(1-\frac{1}{2^4}\right)\left(1-\frac{1}{2^2}\right)\right)\zeta(4)\zeta(2)  \nonumber \\
 &  & +\frac12 \left(1-\frac{1}{2^6}\right)\zeta(6) \nonumber\\
 &=& \left(\frac{1}{15}\right)\frac{\pi^6}{2^6}.\label{Eqn:ZIsos}
\end{eqnarray}
\normalsize
Using (\ref{Eqn:S}), (\ref{Eqn:Se}), and (\ref{Eqn:So}) we have:
\begin{multline}
{\mathcal S}  =  Z_t -\frac{\pi}{2}\left(1-\frac{1}{2^6}\right)\sum_{k\geq 1} \frac{1}{k^5} \coth(k\pi)  \\
+\frac{\pi}{2} \sum_{k\textup{ odd}} \frac{1}{k^5} \coth(k\pi)   - \frac{\pi}{4} \sum_{k\textup{ odd}} \frac{1}{k^5} \coth\left(\frac{k\pi}{2}\right). \nonumber
\end{multline}
\normalsize
Since
\[
\sum_{k\textup{ odd}} \frac{1}{k^5} \coth(k\pi)  = \sum_{k\geq 1} \frac{1}{k^5} \coth(k\pi) - \frac{1}{2^5} \sum_{k\geq 1} \frac{1}{k^5} \coth(2k\pi),  
\]
we have
\begin{multline}
-\left(1-\frac{1}{2^6}\right) \sum_{k\geq 1} \frac{1}{k^5} \coth(k \pi) + \sum_{k\textup{ odd}} \frac{1}{k^5} \coth(k\pi)  = \\ \frac{1}{2^6} \sum_{k\geq 1} \frac{1}{k^5} \left( \coth(k\pi) - 2\coth(2k\pi)\right). \nonumber
\end{multline}
Using the identities
\begin{eqnarray*}
\coth(\theta) & = & \coth(2\theta) + \csch(2\theta) \\
\tanh(\theta) &=& \coth(2\theta) - \csch(2\theta)
\end{eqnarray*}
we have\footnote{There are approximations for the torsional rigidity of a right isosceles triangle in the literature; see for example \cite{YB1}.} 
\begin{proposition}\label{Thm:TRIsosceles}Let $\Omega$ be an isosceles right triangle of side length $L.$  Then
\begin{equation}\label{Eqn:TRIsosceles}
T(\Omega) = \left(\frac12 \right) \frac{4^2L^4}{\pi^6} \left(Z_t -\frac{\pi}{2}\frac{1}{2^6} \sum_{k\geq 1}\frac{1}{k^5} \tanh(k\pi) -  \frac{\pi}{4}\sum_{k \textup{ odd}} \frac{1}{k^5} \coth\left(\frac{k\pi}{2}\right) \right)
\end{equation}
where 
\[
Z_t = \left(\frac{1}{15} \right) \frac{\pi^6}{2^6}. 
\]
\end{proposition}

\section{Proof of Theorem \ref{Thm:DifferentTorsion}}

\begin{figure}
\includegraphics[width=3in,height=3in]{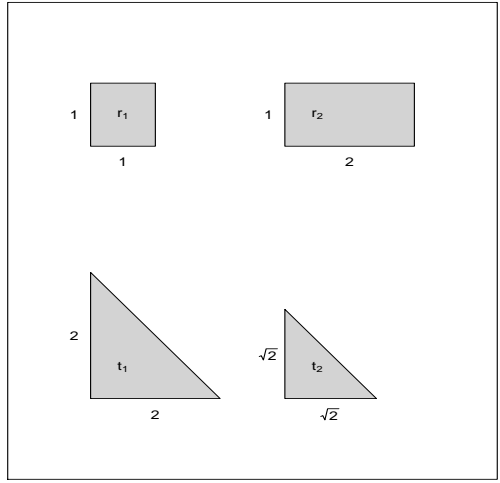}
\caption{Chapman domains $C_1$ and $C_2$ are disjoint unions of a rectangle and a triangle: $C_i = r_i \cup t_i$.}
\label{Fig:Chapman}
\end{figure}

As in figure \ref{Fig:Chapman}, define disjoint planar polygons 
\begin{eqnarray*}
t_1 &=& \textup{ right isosceles triangle of side length } 2 \\
t_2 &=& \textup{ right isosceles triangle of side length } \sqrt{2} \\
r_1 &=& \textup{ square of side length 1} \\
r_2 &=& \textup{ rectangle of length } 2 \textup{ and height } 1.
\end{eqnarray*}

Let $C_1 $ be the disjoint union defined by $C_1 = r_1 \cup t_1$ and let $C_2$ be the disjoint union $C_2=r_2 \cup t_2.$  Because the polygons comprising $C_1$ are disjoint, the Dirichlet spectrum of $C_1$ is a union of the Dirichlet spectra of $r_1$ and $t_1,$ and similarly for $C_2$ (see \cite{Ch1}).  We compute the difference $T(C_1)-T(C_2).$  Using Proposition \ref{Thm:TRIsosceles}, 
\[
T(t_1) -T(t_2) = \left(2^4 -(\sqrt{2})^4\right) \frac{4^2}{2\pi^6} \tau
\]
where 
\[
\tau = \left(\frac{1}{15}\right) \frac{\pi^6}{2^6} -\frac{\pi}{2}\frac{1}{2^6} \sum_{k\geq 1}\frac{1}{k^5} \tanh(k\pi) -  \frac{\pi}{4}\sum_{k \textup{ odd}} \frac{1}{k^5} \coth\left(\frac{k\pi}{2}\right).
\]
Doing the arithmetic:
\begin{equation}\label{Eqn:triang}
T(t_1) -T(t_2) = \frac{1}{10} - \frac{3}{4\pi^5} \sum_{k\geq 1}\frac{1}{k^5} \tanh(k\pi) -\frac{3\cdot 2^3}{\pi^5}\sum_{k \textup{ odd}} \frac{1}{k^5} \coth\left(\frac{k\pi}{2}\right).
\end{equation}
Using Proposition \ref{Thm:FirstMomentR},
\begin{equation}\label{Eqn:rect}
T(r_1) -T(r_2) =  -\frac{1}{12} - \frac{4^2}{\pi^5} \sum_{k \textup{ odd}} \frac{1}{k^5} \left( \tanh\left(\frac{k\pi}{2}\right) -\tanh\left(\frac{k\pi}{4}\right)    \right).
\end{equation}
From (\ref{Eqn:triang}) and (\ref{Eqn:rect}) we have
\[
T(C_1)-T(C_2) = \frac{1}{60} - D 
\]
where $D$ is the sum of the terms involving series, each of which has positive terms.  Note that for $x>0,$ $\coth(x)$ is decreasing and bounded below by 1.  Thus,
\begin{eqnarray*}
\sum_{k \textup{ odd}} \frac{1}{k^5} \coth\left(\frac{k\pi}{2}\right) &=& \sum_{k\geq 1} \frac{1}{k^5} \coth\left(\frac{k\pi}{2}\right) - \frac{1}{2^5}\sum_{k \geq 1} \frac{1}{k^5} \coth\left(k\pi \right) \\
&>& \sum_{k\geq 1} \frac{1}{k^5} \coth\left(\frac{k\pi}{2}\right) - \frac{1}{2^5}\sum_{k \geq 1} \frac{1}{k^5} \coth\left(\frac{k\pi}{2} \right) \\
&>& \left(1-\frac{1}{2^5}\right)\zeta(5).
\end{eqnarray*}
Thus, since $\zeta(5) >1,$
\begin{eqnarray*}
T(C_1)-T(C_2) &<& \frac{1}{60} -\frac{3\cdot 2^3}{\pi^5}\sum_{k \textup{ odd}} \frac{1}{k^5} \coth\left(\frac{k\pi}{2}\right) \\
 &<& \frac{1}{60} -\frac{24}{\pi^5}\left(\frac{31}{32}\right) \\
 & <& 0,
\end{eqnarray*}
which proves Theorem \ref{Thm:DifferentTorsion}.

\begin{corollary}The Chapman domains $C_1$ and $C_2$ are distinguished by heat content.
\end{corollary}

\begin{proof} Torsional rigidity is the first moment of heat content (see \cite{MM1}).  Thus, if $C_1$ and $C_2$ have the same heat content, they must have the same torsional rigidity.

\end{proof}


\begin{thebibliography}{BDK}

\bibitem[BDK]{BDK}
M. van den Berg, E. Dryden, and T. Kappeler, \emph{Isospectrality and heat content}, Bull. Lond. Math. Soc. {\bf 46}, (2014),  793--808. 

\bibitem[Be1]{Be1}
B. C. Berndt, \emph{Ramanujan's Notebooks, Part II}, Springer-Verlag, New York, 1989.

\bibitem[Ch1]{Ch1}
S. J. Chapman \emph{Drums that sound the same}, Amer. Math. Monthly {\bf 102},  (1995), no. 2, 124--138.

\bibitem[CKM]{CKM}
D. Colladay, L. Kaganovskiy and P. McDonald, {\it Torsional rigidity, isospectrality and quantum graphs}, J. Phys. A. {\bf 50}, (2016), no. 3, https://doi.org/10.1088/1751-8121/50/3/035201.

\bibitem[Gi1]{Gi1}
P. Gilkey, \emph{Heat content, heat trace and isospectrality}, New developments in
Lie theory and geometry, Contemp. Math. {\bf 491}, (2009), 115–-123.

\bibitem[Mc1]{Mc1}
B. J. McCartin, \emph{Laplacian Eigenstructure of the Equilateral Triangle}, Hikari Ltd., Bulgaria, 2011.

\bibitem[MM1]{MM1}
P. McDonald and R. Meyers, {\it Dirichlet spectrum and heat content}, J. Functional Anal. {\bf 200}, (2003), 150--159.

\bibitem[MM2]{MM2}
P. McDonald and R. Meyers, {\it Isospectral polygons, planar graphs, and heat content}, Proc. Am. Math. Soc. {\bf 131}, (2003), no. 11, 3589--3599.

\bibitem[P1]{P1}
F. Pockels, \emph{\"Uber die partielle Differentialgleichung $\Delta u +k^2u=0$}, B. G. Teubner, Leipzig, 1891.

\bibitem[Po1]{Po1} G. P\'olya, {\it Torsional rigidity, principal
  frequency, electrostatic capacity and symmetrization,} Quart. Appl. Math. {\bf6}, (1948) 267--277.
  
\bibitem[PS1]{PS1}
G. P\'olya and G.Szeg\"o, \emph{Isoperimetric Inequalities in Mathematical Physics}. Ann. of Math. Stud.
{\bf 27}, Princeton University Press, Princeton (1951).

\bibitem[SV1]{SV1}
M. de Saint-Venant \emph{Memoire sur la torsion des prismes}, Mem. pres. par divers savants Acad. Sci.
Inst. Imperial France Sci. Math. Phys. {\bf 14}, (1856), 233--560,  (2 volumes).

\bibitem[TG1]{TG1}
S. P. Timoshenko and J. N. Goodier, {\it Theory of Elasticity} (third edition), McGraw-Hill Book Co., Singapore, 1970.

\bibitem[YB1]{YB1}
W. C. Young and R. G. Budynas, \emph{Roark's Formulas for Stress and Strain} (seventh edition), McGraw-Hill, New York, 2002.

\end{thebibliography}
\end{document}